\documentclass{amsart}

\usepackage{amsmath}
\usepackage{amsthm}
\usepackage{amsfonts}
\usepackage{amssymb}
\usepackage{mathabx}

\usepackage{pgf,tikz}
\usepackage{mathrsfs}
\usetikzlibrary{arrows}

\usepackage{graphics}
\usepackage{epsfig}
\usepackage{color}
\usepackage{verbatim}

\usepackage[utf8]{inputenc}
\usepackage{enumerate}
\usepackage[bookmarks = true, colorlinks=true, linkcolor = black, citecolor = black, menucolor = black, urlcolor = black]{hyperref}

\newcommand\R{{\mathbb{R}}}
\newcommand\T{{\mathbb{T}}}
\newcommand\C{{\mathbb{C}}}
\newcommand\Z{{\mathbb{Z}}}
\newcommand\N{{\mathbb{N}}}

\newcommand\supp{\operatorname{supp}}

\newcommand\dist{\operatorname{dist}}

\newcommand\CB{{\mathcal B}}

\newcommand\CI{{\mathcal I}}

\newcommand\CM{{\mathcal M}}

\newcommand\CS{{\mathcal S}}

\newcommand\diam{{\operatorname{diam}}}

\theoremstyle{plain}
  \newtheorem{theorem}{Theorem}

  \newtheorem{proposition}[theorem]{Proposition}
  
  \newtheorem{lemma}[theorem]{Lemma}
  \newtheorem{corollary}[theorem]{Corollary}

\theoremstyle{definition}
  \newtheorem{definition}[theorem]{Definition}
  \newtheorem{example}[theorem]{Example}

\date{\today}
\title{Subdyadic square functions and applications to weighted harmonic analysis}
\author{David Beltran}
\author{Jonathan Bennett}
\address{Department of Mathematics, University of Birmingham, Edgbaston Birmingham B15 2TT}
\email{d.beltran@pgr.bham.ac.uk}\email{j.bennett@bham.ac.uk}
\thanks{This work was supported by the European Research Council [grant
number 307617].}

\begin{document}
\begin{abstract}
Through the study of novel variants of the classical Littlewood--Paley--Stein $g$-functions, we obtain pointwise estimates for broad classes of highly-singular Fourier multipliers on $\mathbb{R}^d$ satisfying regularity hypotheses adapted to fine (subdyadic) scales. In particular, this allows us to efficiently bound such multipliers by geometrically-defined maximal operators via general weighted $L^2$ inequalities, in the spirit of a well-known conjecture of Stein.
Our framework applies to solution operators for dispersive PDE, such as the time-dependent free Schr\"odinger equation, and other highly oscillatory convolution operators that fall well beyond the scope of the Calder\'on--Zygmund theory.
\end{abstract}
\maketitle

\section{Introduction}\label{sec:Intro}

A common feature of many themes in both classical and contemporary harmonic analysis is the pivotal role played by  operators which exhibit a certain \emph{quadratic} structure.
Such operators are usually referred to as \emph{square functions}, and their study has its roots in the classical Littlewood--Paley theory (see for example \cite{Stein70SI}, \cite{Stein70LP}, \cite{Stein82}). A common role of the square functions is to capture manifestations of orthogonality in $L^p$ spaces for $p\neq 2$. A primordial example is the square function (or \emph{$g$-function})
\begin{equation}\label{classg}
g(f)(x):=\left(\int_0^\infty \Big|\frac{\partial u}{\partial t}(x,t)\Big|^2tdt\right)^{1/2},
\end{equation}
where $u:\mathbb{R}^d \times \mathbb{R}_+\rightarrow\mathbb{R}$ denotes the Poisson integral of the function $f$ on $\mathbb{R}^d$. While Plancherel's theorem quickly reveals that $\|g(f)\|_2\equiv\|f\|_2$, the key point is that this property essentially persists on $L^p$ -- that is, the norms $\|g(f)\|_p$ and $\|f\|_p$ are equivalent for all $1<p<\infty$. Such facts have many important consequences, making square functions a central tool in modern analysis and PDE. In particular, square functions play a striking role in the classical theory of Fourier multipliers. On an abstract level, this approach to multipliers, originating in fundamental work of Stein \cite{Stein70SI}, consists of identifying square functions $g_1$ and $g_2$ for which we have the \emph{pointwise} estimate
\begin{equation}\label{abstractpointwise}
g_1(T_mf)(x)\lesssim g_2(f)(x);
\end{equation}
here $T_m$ denotes the convolution operator with Fourier multiplier $m$.\footnote{Throughout this paper we shall write $A\lesssim B$ if there exists a constant $c$ such that $A\leq cB$. In particular, this constant will always be independent of the input function $f$, variable $x$ and weight function $w$. The relations $A\gtrsim B$ and $A\sim B$ are defined similarly.}
Given such an estimate one may then deduce bounds on $T_m$ from bounds on the square functions $g_1$ and $g_2$. More specifically, if one has
\begin{equation}\label{abstractbounds}
\|f\|_X\lesssim
\|g_1(f)\|_Y\;\;\mbox{ and }\;\;\|g_2(f)\|_Y\lesssim\|f\|_Z,
\end{equation}
for suitable normed spaces $X,Y,Z$, then the pointwise estimate \eqref{abstractpointwise} quickly reveals that
\begin{equation}\label{mechanism}
\|T_mf\|_X\lesssim\|g_1(T_m f)\|_Y\lesssim\|g_2(f)\|_Y\lesssim\|f\|_Z;
\end{equation}
that is, $T_m$ is bounded from $Z$ to $X$.\footnote{Of course this requires that the norm $\|\cdot\|_Y$ is increasing in the sense that $f_1\lesssim f_2\implies\|f_1\|_Y\lesssim\|f_2\|_Y$.}
The prime example of this approach in action is Stein's celebrated proof of the classical H\"ormander--Mikhlin multiplier theorem, which states that if a Fourier multiplier $m$ on $\mathbb{R}^d$ satisfies
\begin{equation}\label{HormClassic}
\sup_{r>0}\|m(r \cdot) \Psi\|_{H^\sigma} < \infty
\end{equation}
for some $\sigma>d/2$, or equivalently
\begin{equation}\label{HormInh}
\sup_{r>0} r^\theta r^{-d/2} \|m\Psi(r^{-1}\cdot)\|_{\dot{H}^{\theta}} < \infty
\end{equation}
for all $0 \leq \theta \leq \sigma$ and some $\sigma>d/2$, then $m$ is an $L^p$ multiplier for $1<p<\infty$. Here $H^\sigma$ and $\dot{H}^\sigma$ denote the usual inhomogeneous and homogeneous Sobolev spaces of order $\sigma$ respectively, and $\Psi$ is a suitable smooth function with compact support away from the origin.
Stein's proof \cite{Stein70SI} consists of establishing the pointwise estimate
\begin{equation}\label{point}g(T_mf)(x) \lesssim g_\lambda^*(f)(x)\end{equation}
where $g$ is defined in \eqref{classg}, and $g_\lambda^*$ is a further, more robust, variant of $g$ given by
$$
g_\lambda^*(f)(x):=\Big( \int_{\R^{d+1}_+} |\nabla u(y,t)|^2 \Big(1+\frac{|x-y|}{t} \Big)^{-d\lambda} \frac{dydt}{t^{d-1}}\Big)^{1/2}
$$
with $\lambda=2\sigma/d>1$. Since both $g$ and $g_{\lambda}^*$ characterise $L^p$ norms in the sense that
\begin{equation}\label{lpc}
\|g(f)\|_p\sim\|g_\lambda^*(f)\|_p\sim\|f\|_p,
\end{equation}
the pointwise inequality \eqref{point} immediately yields $\|T_mf\|_p\lesssim\|f\|_p$, completing Stein's argument. \footnote{Technically speaking, the equivalence \eqref{lpc} only holds for $2/\lambda<p<\infty$, giving boundedness for $T_m$ when $2\leq p < \infty$. The range $1<p<2$ follows by duality.}
Stein's approach is surprisingly effective also in weighted contexts. For example, if $m$ is a H\"ormander--Mikhlin multiplier on $\mathbb{R}^d$, then by establishing an estimate of the form \eqref{point}, and suitable weighted estimates for square functions closely related to $g$ and $g_\lambda^*$, it follows that
\begin{equation}\label{wilsonest}
\int_{\R^d} |T_mf|^2 w \lesssim \int_{\R^d} |f|^2 M^7w
\end{equation}
for any weight $w$. Here $M$ denotes the classical Hardy--Littlewood maximal operator and $M^7$ its 7-fold composition. It should be noticed that \eqref{wilsonest}, combined with the Hardy--Littlewood maximal theorem, implies the full H\"ormander--Mikhlin multiplier theorem by straightforward duality considerations. Inequality \eqref{wilsonest} is a minor variant of a result of Wilson \cite{Wi89} in the context of certain Calder\'on--Zygmund operators, where such a weighted estimate was obtained via a pointwise estimate of the kind \eqref{abstractpointwise}.\footnote{Unlike in \cite{Wi89}, it is unlikely that the power on the Hardy--Littlewood maximal function is optimal in \eqref{wilsonest}; we do not concern ourselves with such finer points in this paper.} We refer to \cite{Stein70LP}, \cite{CGT}, \cite{Car85}, \cite{LRSsf} for further discussion of Stein's approach.

The purpose of this paper is to show that Stein's approach to proving Fourier multiplier theorems continues to function efficiently far beyond the classical H\"ormander--Mikhlin, Calder\'on--Zygmund or Muckenhoupt $A_p$ theories. The class of multipliers that we consider involves a H\"ormander--Mikhlin-type condition adapted to scales that are much finer than dyadic; such finer scales we shall refer to as \emph{subdyadic}. This class, which originates in work of Miyachi \cite{Mi80,Mi81} (see also H\"ormander \cite{Hormander}), encompasses a wide variety of highly-oscillatory convolution operators on $\mathbb{R}^d$, including solution operators to linear dispersive PDE, such as the time-dependent free Schr\"odinger equation. As a consequence we are able to efficiently bound such multipliers by geometrically-defined maximal operators via general weighted $L^2$ inequalities in the spirit of a well-known conjecture of Stein.

For the sake of clarity we begin by describing the weaker Mikhlin-type formulation of our classes of multipliers. Given $\alpha,\beta \in \R$, consider the class of multipliers $m$ on $\mathbb{R}^d$, with support in the set $\{\xi \in \R^d : |\xi|^\alpha \geq 1\}$, satisfying the Miyachi condition
\begin{equation}\label{MikhlinCondition}
|D^\gamma m(\xi)|\lesssim |\xi|^{-\beta+|\gamma|(\alpha-1)}
\end{equation}
for every multi-index $\gamma \in \N^d$ with $|\gamma| \leq \lfloor \frac{d}{2} \rfloor +1$, where $|\gamma|=\gamma_1+\cdots+\gamma_d$.
This class is modelled by the examples
$
m_{\alpha,\beta}(\xi):=|\xi|^{-\beta}e^{i|\xi|^\alpha}
$, first studied by Hirschman \cite{Hi59}, and later by Wainger \cite{Wainger}, Fefferman \cite{Fe70}, Fefferman and Stein \cite{FSmult}, Miyachi \cite{Mi80,Mi81} and others. We note that these multipliers often correspond to highly-oscillatory convolution kernels; see for example \cite{Sjo81Lp} or the forthcoming Corollary \ref{osccor}. The support condition on $m$ is desirable here, as to impose the \emph{same} power-like behaviour \eqref{MikhlinCondition} as $|\xi|\rightarrow 0$ \emph{and} $|\xi|\rightarrow \infty$ would be artificial, at least for $\alpha\neq0$; for example the specific multiplier $m_{\alpha,\beta}$ naturally satisfies \eqref{MikhlinCondition} for $|\xi|^\alpha\geq 1$, but $|D^\gamma m(\xi)|\lesssim |\xi|^{-\beta-|\gamma|}$ for $|\xi|^\alpha\leq 1$. We refer the reader to Section \ref{sec:Dispersive} for further discussion of multipliers defined on the whole of $\mathbb{R}^d\backslash\{0\}$ satisfying such ``two-sided" conditions. The presence of a distinguished (unit) scale here is indeed quite conventional, as may be seen in the formulation of the standard symbol classes $S^m$; see for example \cite{bigStein}. The advantage of imposing a support condition rather than a global estimate of the form $|D^\gamma m (\xi)|\lesssim (1+|\xi|)^{-\beta+|\gamma|(\alpha-1)}$ is that it also has content for $\alpha<0$.

Our results will naturally apply to broader classes of multipliers than that given by the pointwise condition \eqref{MikhlinCondition}. For example, it will be enough to ask that the multiplier $m$, with support in $\{\xi \in \R^d : |\xi|^\alpha \geq 1\}$, satisfies the weaker condition
\begin{equation}\label{HormSD}
\sup_{B}\; \dist(B,0)^{\beta+(1-\alpha)|\gamma|} \Big( \frac{1}{|B|}\int_B |D^\gamma m(\xi)|^2 d\xi \Big)^{1/2}< \infty
\end{equation}
for all $\gamma \in \N^d$ with $|\gamma|\leq \lfloor \frac{d}{2} \rfloor +1$. Here the supremum is taken over all euclidean balls $B$ in $\mathbb{R}^d$ with $\dist(B,0)^\alpha\geq 1$
such that $$r(B) \sim \dist(B,0)^{1-\alpha},$$ where $r(B)$ denotes the radius of $B$. Observe that for $\alpha\neq 0$, typically $r(B)\ll\dist(B,0)$, making it natural to refer to such balls as \emph{subdyadic} (or \textit{$\alpha$-subdyadic}). As may be expected, the condition \eqref{HormSD} may be weakened still further to
\begin{equation}\label{HormSobSD}
\sup_{B} \dist(B,0)^{\beta+(1-\alpha)\theta} |B|^{-1/2}\|m\Psi_B\|_{\dot{H}^\theta}<\infty,
\end{equation}
for all $0\leq \theta\leq \sigma$ and some $\sigma>d/2$,
uniformly over normalised bump functions $\Psi_B$ adapted to an $\alpha$-subdyadic ball $B$ (see Section \ref{sec:Horm} for the definition of a normalised bump function and further clarification).
This is easily seen to reduce to the classical H\"ormander condition \eqref{HormInh} when $\alpha=\beta=0$.
Observe that \eqref{MikhlinCondition} implies \eqref{HormSD}, which in turn implies the more general condition \eqref{HormSobSD} by the Leibniz formula.

In the context of these multipliers we introduce the square function
\begin{equation}\label{subsquare}
g_{\alpha,\beta}(f)(x):=\Big(\int_{\Gamma_\alpha(x)} |f \ast \phi_t(y)|^2 \frac{dy}{t^{(1-\alpha)d+2\beta}} \frac{dt}{t}\Big)^{1/2},
\end{equation}
where $\phi$ is a smooth function with suitable compact Fourier support away from the origin, $\phi_t(x):=t^{-d}\phi(x/t)$ for $t>0$, and
$$
\Gamma_\alpha(x):=\{(y,t)\in \R^{d}\times\R_+: t^\alpha \leq 1, |y-x| \leq t^{1-\alpha}\}.
$$
For $\alpha\neq 0$ the region $\Gamma_\alpha(x)$ is very different from the classical cone $\Gamma_0(x)$. In particular, when $\alpha>1$ it becomes an ``inverted cone", allowing approach to infinite order, and when $\alpha<0$, it is perhaps best interpreted as an ``escape" region since $t\geq 1$; see \cite{Ben2014} for further discussion of these regions in the context of maximal operators and the dimension $d=1$.

By close analogy with the classical $g_\lambda^*$ we also introduce the more robust square function
$$
g_{\alpha,\beta,\lambda}^*(f)(x)=\Big(\int_{t^\alpha \leq 1} \int_{\R^d} |f \ast \phi_t (y)|^2 \Big(1+\frac{|x-y|}{t^{1-\alpha}} \Big)^{-d\lambda} \frac{dy}{t^{(1-\alpha)d+2\beta}}\frac{dt}{t}\Big)^{1/2},
$$
which is manifestly a pointwise majorant of
$g_{\alpha,\beta}$. It should be observed that $g_{0,0}$ and $g_{0,0,\lambda}^*$ are minor variants of the classical $g$ and $g_{\lambda}^*$, where $\partial_t P_t$ and $\nabla P_t$ are formally replaced by $t\phi_t$. Here $P_t$ denotes the Poisson kernel.
As will become clear later, the ``local" nature of these square functions in the $t$-integral is related to the support property of the multipliers $m$ described above.

The following three theorems capture Stein's approach in this setting, with the more novel aspects appearing in the first two. Our main application, from which all subsequent applications essentially derive, is contained in Corollary \ref{MikhlinThm}.
\begin{theorem} \label{PointwiseThm}
Let $\alpha,\beta \in \R$ and $m$ be a multiplier satisfying \eqref{HormSobSD}. Then
\begin{equation}\label{pointwise}
g_{\alpha,\beta}(T_mf)(x) \lesssim g_{\alpha,0,\lambda}^*(f)(x),
\end{equation}
with $\lambda=2\sigma/d>1$.
\end{theorem}
We note that it is not necessary to impose a support condition on the multiplier $m$ in Theorem \ref{PointwiseThm} thanks to the Fourier support property of the function $\phi$ in the definition of $g_{\alpha,\beta}$. This Fourier support property prevents $g_{\alpha,\beta}(T_mf)$ from detecting the portion of multiplier $m$ supported in $\{\xi \in \R^d: |\xi|^\alpha \leq 1\}$; we refer to the end of Section \ref{sec:Horm} for full details of the relevant properties of $\phi$, and further clarification of this technical point.

The pointwise estimate \eqref{pointwise} has an even broader context in classical and contemporary harmonic analysis. Indeed pointwise estimates of the type \eqref{abstractpointwise} may be found in situations where the square functions $g_1$ and $g_2$ are replaced by other sorts of auxiliary operators. A classical example in the setting of Calder\'on--Zygmund theory, due to Córdoba and Fefferman \cite{CF}, replaces $g_1$ and $g_2$ by the \emph{sharp maximal function} $M^\#$ and (essentially) the Hardy--Littlewood maximal function $M$ respectively. More sophisticated variants of this estimate have proved highly effective very recently, for instance, playing a central role in Lerner's alternative proof of the $A_2$-conjecture \cite{LeA2}; see also Hyt\"onen \cite{Hyt2012}. Further developments in this direction have led to even stronger pointwise estimates for Calder\'on--Zygmund operators, where the first auxiliary operator $g_1$ is entirely absent; see Lacey \cite{Lac2015} and the references there. Our Theorem \ref{PointwiseThm}, being strongly oscillatory in nature (see the forthcoming Corollary \ref{osccor} for instance), is perhaps closer in perspective to the pointwise estimates obtained by Grafakos, Martell and Soria
\cite{GMS} for the Carleson operator, and by Bernicot \cite{Ber2014} for multi-frequency Calder\'on--Zygmund operators.

Theorem \ref{PointwiseThm}, along with the general mechanism \eqref{mechanism}, allows one to find bounds for the multipliers \eqref{MikhlinCondition} provided suitable forward and reverse bounds for $g_{\alpha,0,\lambda}^*$ and $g_{\alpha,\beta}$ (respectively) may be found. Of these our reverse bound is the more interesting since it involves a natural maximal function analogue of $g_{\alpha,\beta}$.
\begin{theorem}\label{ReverseThm}
Let $\alpha,\beta \in \R$, and $f$ be a function such that $\supp(\widehat{f})\subseteq \{\xi \in \R^d : |\xi|^\alpha \geq 1\}$. Then
\begin{equation*}
\int_{\R^d}|f|^2w \lesssim \int_{\R^d} g_{\alpha,\beta}(f)^2 \CM_{\alpha,\beta}M^4w
\end{equation*}
for any weight $w$, where
$$
\CM_{\alpha,\beta} w(x)= \sup_{(y,r) \in \Gamma_\alpha(x)} \frac{1}{|B(y,r)|^{1-2\beta/d}} \int_{B(y,r)} w.
$$
\end{theorem}
It is interesting to compare $g_{\alpha,\beta}$ with Rubio de Francia's square function for arbitrary intervals \cite{RdF}. Theorem \ref{ReverseThm} for $(\alpha,\beta)\neq (0,0)$ may be viewed as a substitute for the lack of reverse $L^p$ estimates in \cite{RdF} when $2\leq p<\infty$; see Section \ref{sec:MaxOp} for further discussion. The reader may also wish to compare Theorem \ref{ReverseThm} with the classical duality between nontangential maximal operators and Carleson measures in \cite{bigStein}, and the recent $L^p$ theory for outer measures of Do and Thiele \cite{DT}.

The maximal operator $\CM_{\alpha,\beta}$ may be interpreted as a fractional Hardy--Littlewood maximal operator associated with the region $\Gamma_{\alpha}$; we refer to \cite{Ben2014} for some further discussion in the one-dimensional case. Naturally, the case $\alpha=0$ corresponds to the classical fractional Hardy--Littlewood maximal function. For $0<\alpha<1$ the maximal functions $\CM_{\alpha,\beta}$ are closely-related to those considered by Nagel and Stein \cite{NS} in a different context. In dimensions larger than one, the maximal operators $\mathcal{M}_{\alpha,\beta}$ are relatives of the \emph{Nikodym} (or \emph{Kakeya}) maximal operators. In particular, elementary considerations reveal the pointwise bound
\begin{equation}\label{NikKak}
\mathcal{M}_{\alpha, \beta}f\gtrsim\mathcal{N}_{\alpha,\beta}f,
\end{equation}
where
$$
\mathcal{N}_{\alpha,\beta}f(x):=\sup_{0<r^{\alpha}\leq 1}\sup_{\substack{T\in\mathcal{T}_\alpha(r)\\T\ni x}}\frac{r^{2\beta}}{|T|}\int_T|f|
$$
and $\mathcal{T}_{\alpha}(r)$ denotes the collection of cylindrical tubes $T$ of length $r^{1-\alpha}$ and cross-sectional radius $r$ in $\mathbb{R}^d$. The inequality \eqref{NikKak} follows merely by covering each $T\in\mathcal{T}_{\alpha}(r)$ by $O(r^{-\alpha})$ balls of radius $r$, and noting their positions.

The forward estimate for $g_{\alpha,\beta,\lambda}^*$ is more classical in nature than its reverse counterpart above. We state it here as a theorem for the sake of structural coherence.
\begin{theorem}\label{ForwardThm}
Let $\alpha \in \R$ and $\lambda>1$. Then
\begin{equation*}
\int_{\R^d} g_{\alpha,0,\lambda}^*(f)^2 w \lesssim \int_{\R^d} |f|^2 M^2w.
\end{equation*}
\end{theorem}
Combining Theorems 1--3 via \eqref{mechanism} immediately leads to the following weighted estimates for $T_m$.
\begin{corollary}\label{MikhlinThm}
Let $\alpha,\beta \in \R$ and $m$ be a multiplier supported in $\{ \xi \in \R^d : |\xi|^\alpha \geq 1\}$ satisfying
\eqref{HormSobSD}. Then
\begin{equation}\label{mainweight}
\int_{\R^d} |T_m f|^2 w \lesssim \int_{\R^d} |f|^2 M^2\CM_{\alpha,\beta}M^4w.
\end{equation}
\end{corollary}
Corollary \ref{MikhlinThm} provides us with an opportunity to comment on the optimality of Theorems \ref{PointwiseThm} -- \ref{ForwardThm}, since \eqref{mainweight}, along with elementary duality considerations, allows us to transfer $L^p-L^q$ bounds for $\mathcal{M}_{\alpha,\beta}$ to such bounds for $T_m$. As is easily verified, the optimal bounds on $\mathcal{M}_{\alpha,\beta}$ (see Section \ref{sec:MaxOp}) may be reconciled with the optimal bounds on the specific multipliers $m_{\alpha,\beta}$ (see Miyachi \cite{Mi80}) in this way.

We now turn to some rather less immediate applications of Theorems \ref{PointwiseThm} -- \ref{ForwardThm}.

\subsection{Oscillatory kernels} The method of stationary phase allows one to use Theorems \ref{PointwiseThm} -- \ref{ForwardThm} to obtain similar pointwise and general-weighted estimates for classes of highly oscillatory convolution kernels. For example we have the following:
\begin{corollary}\label{osccor}
For $a>0$, $a\neq 1$ and $b \geq d(1-\frac{a}{2})$, let $K_{a,b}:\R^d \to \C$ be given by
$$
K_{a,b}(x)=\frac{e^{i|x|^a}}{|x|^b}(1-\eta(x)),
$$
where $\eta \in C^\infty_c(\R^d)$ is such that $\eta(x)=1$ for all $x$ belonging to a neighbourhood of the origin. Then for any $\lambda>0$,
\begin{equation}\label{pointwiseosc}
g_{\alpha,\beta}(K_{a,b}*f)(x)\lesssim g_{\alpha,0,\lambda}^*(f)(x)
\end{equation}
and
\begin{equation}\label{mainosc}
\int_{\R^d} |K_{a,b} \ast f|^2 w \lesssim \int_{\R^d} |f|^2 M^2\CM_{\alpha,\beta}M^4w,
\end{equation}
where $\alpha=\frac{a}{a-1}$ and $\beta=\frac{da/2-d+b}{a-1}$.
\end{corollary}
This corollary is based on the fact, established in Sj\"olin \cite{Sjo81Lp}, that the Fourier transform of $K_{a,b}$ satisfies the pointwise bound \eqref{MikhlinCondition} for the indicated parameters $\alpha$, $\beta$, and all multi-indices $\gamma$; see Section \ref{sec:Cor}.

Inequality \eqref{mainweight} (see also \eqref{mainosc}), which includes the more classical \eqref{wilsonest} as a special case, is striking as it involves the control of a highly cancellative operator by a geometrically-defined (and positive) maximal function. This result bears a very close resemblance to a longstanding conjectured inequality in the context of the \emph{disc multiplier}, raised in the 1970s by Stein \cite{Stein79} (see also C\'ordoba \cite{CordKak} in the general context of the Bochner--Riesz multipliers), which is often referred to as \emph{Stein's conjecture}. In \cite{Stein79}, Stein conjectured that a weighted inequality of the general form \eqref{mainweight} should hold for the disc multiplier. In this case, the conjectured controlling maximal function is some variant of the \emph{universal maximal function}
$$
\mathcal{N}w(x):=\sup_{T\ni x}\frac{1}{|T|}\int_T w,
$$
where the supremum is taken over all rectangles $T$ containing the point $x$.
This question is far from having a satisfactory answer for $d\geq 2$; notice that the associated convolution kernel
\begin{equation}\label{BRKernel}
K(x):=\mathcal{F}^{-1}m(x)=c\frac{e^{2\pi i|x|}+e^{-2\pi i|x|}+o(1)}{|x|^{\frac{d+1}{2}}}
\end{equation}
is far from being integrable and is of course \emph{oscillatory}. Perhaps remarkably, this kernel takes the form of the missing endpoint case $a=1$ in Corollary \ref{osccor}, although it should be noted that the behaviour of these kernels is notoriously discontinuous there.
We refer the reader to the work of Carbery \cite{CarWeight}, Christ \cite{Christ}, Carbery and Seeger \cite{CS}, Carbery, Romera and Soria \cite{CRS}, Carbery, Soria, Vargas and the second author \cite{BCSV} or Lee, Rogers and Seeger \cite{LRSw} for further results in the direction of Stein's conjecture. Other recent examples of weighted inequalities and their controlling maximal operators in different oscillatory contexts may be found in the work of Harrison and the second author \cite{BH}, \cite{Ben2014}, C\'ordoba and Rogers \cite{CR2014} or \cite{Bel2015}.
\subsection{Dispersive and wave-like equations}
A further application of Theorems \ref{PointwiseThm} -- \ref{ForwardThm} in the setting of the specific multipliers $m_{\alpha,\beta}(\xi):=|\xi|^{-\beta}e^{i|\xi|^\alpha}$, yields general-weighted estimates for the solution $u(x,s)=e^{is(-\Delta)^{\alpha/2}}f(x)$ of the dispersive or wave-like equation
\begin{equation}\label{IVP}
\begin{cases}
i \partial_s u +(-\Delta)^{\alpha/2}u=0 \\
u(\cdot,0)=f.
\end{cases}
\end{equation}
For example, we have the following corollary, which we establish in Section \ref{sec:Dispersive}; see also the forthcoming Corollary \ref{inhomPDEthm} involving inhomogeneous derivatives.
\begin{corollary}\label{PDEcorGlob}
\begin{equation}\label{SchrodingerEst}
\int_{\R^d} |e^{is(-\Delta)^{\alpha/2}}f|^2w \lesssim \int_{\R^d} |(-\Delta)^{\beta/2} f|^2 M^2\mathfrak{M}_{\alpha,\beta}^sM^4w,
\end{equation}
where
$$
\mathfrak{M}_{\alpha,\beta}^s w(x):= \sup_{(y,r) \in \Lambda_\alpha^s(x)} \frac{1}{|B(y,r)|^{1-2\beta/d}} \int_{B(y,r)} w,
$$
and $$\Lambda_\alpha^s(x):=\{(y,r)\in\mathbb{R}^d\times\mathbb{R}_+:|x-y|\leq sr^{1-\alpha}\}.$$
\end{corollary}
It should be observed that Corollary \ref{PDEcorGlob}, combined with the trivial uniform $L^{d/(2\beta)}\rightarrow L^\infty$ bound on $\mathfrak{M}_{\alpha,\beta}^s$ and a straightforward duality argument (see Section \ref{sec:MaxOp}), quickly recovers the elementary sharp homogeneous Strichartz inequality
$$
\|e^{is(-\Delta)^{\alpha/2}}f\|_{L_s^\infty L_x^q}\lesssim\|f\|_{\dot{H}^\beta};\;\;\beta=d\left(\frac{1}{2}-\frac{1}{q}\right),\;\;\;\;2 \leq q < \infty,
$$
which follows by Sobolev embedding and energy conservation.

Inequality \eqref{SchrodingerEst} may be regarded as a ``local energy estimate" that also captures dispersive effects of the propagator $e^{is(-\Delta)^{\alpha/2}}$ via the $s$-evolution of the region $\Lambda_\alpha^s(x)$. Indeed the sets $\Lambda_\alpha^s(x)$ are increasing in $s$, so that, in particular
\begin{equation}\label{maximal}
\sup_{0<s\leq 1}\int_{\mathbb{R}^d}
|e^{is(-\Delta)^{\alpha/2}}f|^2w\lesssim\int_{\mathbb{R}^d}|(-\Delta)^{\beta/2}f|^2M^2\mathfrak{M}_{\alpha,\beta}M^4w,
\end{equation}
where $\mathfrak{M}_{\alpha,\beta}:=\mathfrak{M}_{\alpha,\beta}^1$.
It is interesting to compare this inequality with the weighted maximal estimates in \cite{BBCR} (or \cite{Mattila}) at the interface with geometric measure theory -- an exercise that raises the possibility that, for $\beta$ beyond some critical threshold, \eqref{maximal} may be strengthened to
$$
\int_{\mathbb{R}^d}
\sup_{0<s\leq 1}|e^{is(-\Delta)^{\alpha/2}}f|^2w\lesssim\int_{\mathbb{R}^d}|(-\Delta)^{\beta/2}f|^2\mathfrak{M}_{\alpha,\beta}w,
$$ modulo suitable factors of $M$; see for example Rogers and Seeger \cite{RS2010} for related estimates in an unweighted setting. We do not pursue this further here.

\subsubsection*{Structure of the paper} In Section \ref{sec:Horm} we clarify the definitions of the objects introduced in this section, and describe how the subdyadic balls interacts with the square functions $g_{\alpha,\beta}$ and $g_{\alpha,\beta,\lambda}^*$. We prove Theorem \ref{PointwiseThm} in Section \ref{sec:Pointwise}, followed by Theorems \ref{ReverseThm} and \ref{ForwardThm} in Section \ref{sec:reverse}. Our applications to oscillatory kernels and dispersive equations appear in Sections \ref{sec:Cor} and \ref{sec:Dispersive}. In Section \ref{sec:MaxOp} we address questions of optimality by establishing Lebesgue space bounds for $g_{\alpha,\beta}$ and $\mathcal{M}_{\alpha,\beta}$. Finally, we collect together some aspects of classical Littlewood--Paley theory that we appeal to in an appendix.

\subsubsection*{Acknowledgments} We thank Teresa Luque and Maria Reguera for contributions at the early stages of this work.
We also thank Tony Carbery, Sebastian Herr, Keith Rogers, Brian Taylor and Jim Wright for stimulating discussions.
\section{The subdyadic decomposition}\label{sec:Horm}
The classical square functions $g_{0,0}$ and $g_{0,0,\lambda}^*$, associated with the standard conical region $\Gamma_0(x)=\{(y,t)\in\mathbb{R}^d\times\R_+:|x-y|\leq t\}$, are commonly described as \emph{dyadic} as they are able to detect ``orthogonality across dyadic scales", but effectively no finer. This is manifested in the ``decouplings"
\begin{equation}\label{dyadicdecoupling}
g_{0,0}\left(\sum S_{2^k}f\right)^2(x)\lesssim\sum g_{0,0,\lambda}^*(S_{2^{k}}f)^2(x)\lesssim g_{0,0,\lambda}^*\left(\sum S_{2^k}f\right)^2(x),
\end{equation}
where $S_{2^k}$ is the frequency projection onto the dyadic annulus $A_k=\{\xi\in\mathbb{R}^d:|\xi|\sim 2^k\}$. As we shall see, the square functions $g_{\alpha,\beta}$ and $g_{\alpha,\beta, \lambda}^*$, which we refer to as \emph{subdyadic} when $\alpha\neq 0$, detect orthogonality across subdyadic scales, leading to a decoupling of the form \eqref{dyadicdecoupling} associated with suitable families $\CB$ of subdyadic balls. This will play a crucial role in our proof of Theorem \ref{PointwiseThm}.

In order to describe this decoupling we must first clarify the notion of an \emph{$\alpha$-subdyadic ball}, and introduce an appropriate partition of unity adapted to such balls.
\begin{definition} Let $\alpha\in\R$.
A euclidean ball $B$ in $\R^d$ is \emph{$\alpha$-subdyadic} if $\dist(B,0)^\alpha\geq 1$ and
\begin{equation}\label{def:sub} r(B) \sim \dist(B,0)^{1-\alpha},\end{equation}
where $r(B)$ denotes the radius of $B$.
\end{definition}
Of course the definition above depends on the implicit constants in \eqref{def:sub}. Throughout this paper these are taken to be fixed, although chosen so that $\dist(\widetilde{B},0)\sim\dist(B,0)$
where $\widetilde{B}$ denotes the concentric double of $B$.

Before introducing our partition of unity, let us clarify the condition \eqref{HormSobSD}. Following Stein \cite{bigStein}, by a \emph{normalised bump function}
we mean a smooth function $\Psi$ in $\R^d$, supported in the unit ball, such that
$
\|D^\gamma\Psi\|_\infty\leq 1
$
for all multi-indices $\gamma$ with $|\gamma|\leq N$. Here $N$ is a fixed large number, which for our purposes should be taken to exceed $d$. Given a euclidean ball $B$ in $\R^d$, a \emph{normalised bump function adapted to $B$} is a function of the form $\Psi_B:=\Psi\circ A_B^{-1}$, where $\Psi$ is a normalised bump function and $A_B$ is the affine transformation mapping the unit ball onto $\widetilde{B}$.

Let $\CB$ be a family of $\alpha$-subdyadic balls $B$, with bounded overlap, and supporting a regular partition of unity $\{\widehat{\psi}_B\}_{B \in \CB}$ on $\{ |\xi|^\alpha \geq 1 \}$. By regular we mean that $\supp (\widehat{\psi}_B) \subseteq \widetilde{B}$ and
\begin{equation}\label{regcond}
|D^\gamma \widehat{\psi}_B(\xi)|\lesssim r(B)^{-|\gamma|}
\end{equation}
for all multi-indices $\gamma$ with $|\gamma| \leq N$, uniformly in $B$; for technical reasons that will become apparent later, we shall actually assume that $\supp(\widehat{\psi}_B)$ is contained in a concentric dilate of $B$ with some fixed dilation factor strictly less that $2$. This partition of unity gives rise to the reproducing formula
\begin{equation}\label{reproducing}
f=\sum_{B \in \CB} f \ast \psi_B,
\end{equation}
whenever $\supp(\widehat{f}) \subseteq \{\xi\in\R^d:|\xi|^\alpha \geq 1\}$. Observe that the case $\alpha=0$ corresponds to a classical dyadic frequency decomposition; an example of such a family $\CB$ is obtained by decomposing $\R^d$ into the dyadic annuli $A_k$ and covering each $A_k$ by boundedly many balls of radius $2^k$. For general $\alpha$ and $\CB$, elementary geometric considerations reveal that each dyadic annulus $A_k$ will be covered by $O(2^{\alpha dk})$ elements of $\CB$ of radius $O(2^{(1-\alpha)k})$. The following explicit ``lattice-based" example of a cover $\CB$ and partition $\{\widehat{\psi}_B\}$ will be of use to us later on.
\begin{example}\label{Example}
Let $\eta\in\mathcal{S}(\mathbb{R}^d)$ have Fourier support in the annulus $\{|\xi|\sim 1\}$ and be such that $$\sum_{k\in\mathbb{Z}}\widehat{\eta}_k(\xi)=1$$ for all $\xi\neq 0$, where $\widehat{\eta}_k(\xi):=\widehat{\eta}(2^{-k}\xi)$. Thus $\{\widehat{\eta}_k\}$ forms a partition of unity on $\mathbb{R}^d\backslash\{0\}$ with $\supp(\widehat{\eta}_k)\subseteq\{|\xi|\sim 2^k\}$ for each $k\in\mathbb{Z}$. Next let $\nu\in\mathcal{S}(\mathbb{R}^d)$ have Fourier support in $\{|\xi|\lesssim 1\}$ be such that $$\sum_{\ell\in\mathbb{Z}^d}\widehat{\nu}(\xi-\ell)=1$$ for all $\xi\in\mathbb{R}^d$. For each $k\in\mathbb{Z}$ and $\ell\in\mathbb{Z}^d$ let $\widehat{\nu}_k(\xi):=\widehat{\nu}(2^{-(1-\alpha)k}\xi)$ and $\widehat{\nu}_{k,\ell}(\xi):=\widehat{\nu}_k(\xi-2^{(1-\alpha)k}\ell)$. Thus for $\widehat{\zeta}_{k,\ell}(\xi):=\widehat{\eta}_k(\xi)\widehat{\nu}_{k,\ell}(\xi)$ we have $$\sum_{\ell\in\mathbb{Z}^d}\sum_{k\in\mathbb{Z}}\widehat{\zeta}_{k,\ell}=1$$ on $\{|\xi|^\alpha \geq 1\}$. 
Finally we choose, as we may, a family of balls $\CB$
and functions $\{\psi_B\}$ so that for each $B\in\CB$ there is $(k,\ell)\in\mathbb{Z}\times\mathbb{Z}^d$ for which $\psi_B=\zeta_{k,\ell}$ and $\diam(\supp(\widehat{\zeta}_{k,\ell}))\sim r(B)$. By construction $\{\widehat{\psi}_B\}$ forms a partition of unity on $\{|\xi|^\alpha \geq 1\}$ of the type required, provided the implicit constants are chosen suitably.
\end{example}

As we shall see in Section \ref{sec:Pointwise}, the square functions $g_{\alpha,\beta}$ and $g_{\alpha,\beta,\lambda}^*$ decouple such dyadic frequency decompositions, for $\lambda>1$. In particular, we shall prove that
\begin{equation}\label{decoupling}
g_{\alpha,\beta}\Big(\sum_{B \in \CB} f \ast \psi_B \Big)(x)^2 \lesssim \sum_{B \in \CB} g_{\alpha,\beta,\lambda}^*(f \ast \psi_B)(x)^2
\end{equation}
and, at least for the specific partition given in Example \ref{Example},
\begin{equation}\label{recoupling}
\sum_{B \in \CB} g_{\alpha,\beta,\lambda}^*(f \ast \psi_B)(x)^2 \lesssim g_{\alpha,\beta,\lambda}^* \Big( \sum_{B \in \CB}f \ast \psi_B \Big)(x)^2.
\end{equation}
These decoupling and re-coupling inequalities, together with the reproducing formula \eqref{reproducing}, immediately reduce the proof of Theorem \ref{PointwiseThm} to functions localised at a subdyadic frequency scale, that is
\begin{equation}\label{pointlocal}
g_{\alpha,\beta,\lambda}^*(T_m(f \ast \psi_B))(x) \lesssim g_{\alpha,0,\lambda}^*(f \ast \psi_B)(x)
\end{equation}
uniformly in $B\in\CB$. 

We end this section with some clarifying remarks on the function $\phi$ in the definitions of $g_{\alpha,\beta}$ and $g_{\alpha,\beta,\lambda}^*$, and in particular the location of its compact Fourier support.  For technical reasons it will be convenient, as with a number of our constructions, to require this support to depend on the sign of $\alpha$. Specifically, let us suppose that $\widehat{\phi}$ is supported in $\{1 \leq |\xi| \leq 2\}$ for $\alpha\geq 0$, and in $\{1/2 \leq |\xi| \leq 1\}$ for $\alpha<0$. The main purpose of this is to ensure that $g_{\alpha,\beta}(f)\equiv 0$ whenever $\widehat{f}$ is supported in $\{\xi\in\mathbb{R}^d:|\xi|^\alpha\leq 1\}$. This feature, which also relies on the restriction $t^\alpha\leq 1$ in the definition of $\Gamma_{\alpha}(x)$,  makes $g_{\alpha,\beta}$ well-adapted to the support hypothesis imposed on the multipliers that we consider. In particular, we have that $g_{\alpha,\beta}(T_mf)\equiv g_{\alpha,\beta}(T_{m'}f)$ whenever $m$ and $m'$ agree on $\{\xi\in\mathbb{R}^d: |\xi|^\alpha\geq 1\}$, a property that will be convenient in our applications in Sections \ref{sec:Cor} and \ref{sec:Dispersive}. Finally, as we clarify in the appendix, it is also natural to assume that the function $\phi$ satisfies the normalisation condition
\begin{equation}\label{ContinuousPartition}
\int_0^\infty \widehat{\phi}(t\xi)\frac{dt}{t}=1;\;\;\; \xi\neq 0.
\end{equation}
\section{Proof of Theorem \ref{PointwiseThm}: the pointwise estimate}\label{sec:Pointwise}
As indicated in Section \ref{sec:Horm}, our proof of Theorem \ref{PointwiseThm} may be reduced to proving the inequalities \eqref{decoupling}, \eqref{recoupling} and \eqref{pointlocal}. These we establish in turn below.
\subsection{Decoupling subdyadic frequency decompositions}
Before proceeding with the proof of the decoupling estimate \eqref{decoupling}, we need to introduce the auxiliary square function
$$
g_{\alpha,\beta,\Phi}(f)(x)=\Big(\int_{t^\alpha \leq 1} \int_{\R^d} |f \ast \phi_t (y)|^2 \Phi\Big(\frac{x-y}{t^{1-\alpha}} \Big) \frac{dy}{t^{(1-\alpha)d+2\beta}}\frac{dt}{t}\Big)^{1/2},
$$
where $\Phi$ is a Schwartz function such that $\Phi(x)\geq c$ for $|x|\leq 1$ and $\supp(\widehat{\Phi})\subseteq \{\xi \in \R^d: |\xi|\leq 1\}$.\footnote{Observe that such a function $\Phi$ can be constructed by $\Phi=|\Theta|^2 \geq 0$, with $\Theta\in\mathcal{S}(\mathbb{R}^d)$ satsifying $\Theta(0)\neq 0$ and $\supp (\widehat{\Theta})$ compact.} Note that, up to constant factors, $g_{\alpha,\beta, \Phi}$ is a pointwise majorant of $g_{\alpha,\beta}$, and is pointwise majorised by $g_{\alpha,\beta,\lambda}^*$ for any $\lambda>0$.

By \eqref{reproducing} we have
\begin{align*}
g_{\alpha,\beta,\Phi}(f)(x)^2 =\int_{t^\alpha \leq 1} \int_{\R^d} \Bigl|\sum_{B \in \CB} \psi_B \ast \phi_t \ast f (y)\Bigr|^2 \Phi\Big(\frac{x-y}{t^{1-\alpha}}\Big) \frac{dy}{t^{(1-\alpha)d+2\beta}} \frac{dt}{t}.
\end{align*}
On multiplying out the square and using the Fourier transform, the inner (spatial) integral in this expression becomes
\begin{eqnarray*}
\begin{aligned}
\int_{\R^d} &\sum_{B,B'\in\CB} (\psi_B \ast \phi_t \ast f) (y)\overline{(\psi_{B'} \ast \phi_t \ast f) (y)} \Phi \Big(\frac{x-y}{t^{1-\alpha}} \Big) \frac{dy}{t^{(1-\alpha)d+2\beta}}\\
&=\int_{\R^d} \sum_{B,B'\in\CB} \int_{\R^d} \int_{\R^d} \widehat{\psi}_B(\xi) \overline{\widehat{\psi}_{B'}(\eta)}\widehat{\phi}(t\xi)\overline{\widehat{\phi}(t\eta)}\widehat{f}(\xi)\overline{\widehat{f}(\eta)}e^{i y \cdot (\xi-\eta)} \Phi \Big(\frac{x-y}{t^{1-\alpha}} \Big) d\xi d\eta\frac{dy}{t^{(1-\alpha)d+2\beta}} \\
&= \sum_{B,B'\in\CB} \int_{\R^d} \int_{\R^d} \widehat{\psi}_B(\xi) \overline{\widehat{\psi}_{B'}(\eta)})\widehat{\phi}(t\xi)\overline{\widehat{\phi}(t\eta)}\widehat{f}(\xi)\overline{\widehat{f}(\eta)} e^{ix \cdot(\xi-\eta)} \widehat{\Phi}(t^{1-\alpha}(\xi-\eta)) d\xi d\eta \frac{1}{t^{2\beta}}.
\end{aligned}
\end{eqnarray*}
The support conditions on $\widehat{\phi}$, $\widehat{\psi}_B$ and $\widehat{\Phi}$ ensure that the summand above vanishes unless $r(B)\sim r(B')\sim t^{\alpha-1}$ and $\dist(B,B')\lesssim t^{\alpha-1}$. In particular, since $\CB$ consists of balls of bounded overlap, for each such $B$ there are boundedly many $B'$ satisfying these constraints. Consequently,
\begin{align*}
g_{\alpha,\beta,\Phi}(f)(x)^2 & = \int_{t^\alpha \leq 1} \int_{\R^d} \sum_{\substack{B, B' \in \CB \\ r(B)\sim r(B')\sim t^{\alpha-1}\\\dist(B,B')\lesssim t^{\alpha-1}}} (\psi_B \ast \phi_t \ast f) (y)\overline{(\psi_{B'} \ast \phi_t \ast f) (y)} \Phi \Big(\frac{x-y}{t^{1-\alpha}} \Big) \frac{dy}{t^{(1-\alpha)d+2\beta}}\frac{dt}{t},
\end{align*}
which by the Cauchy--Schwarz inequality yields
\begin{align*}
g_{\alpha,\beta,\Phi}(f)(x)^2 & \lesssim \int_{t^\alpha \leq 1} \int_{\R^d} \sum_{B \in \CB} |\psi_B \ast \phi_t \ast f (y)|^2 \Phi \Big(\frac{x-y}{t^{1-\alpha}} \Big) \frac{dy}{t^{(1-\alpha)d+2\beta}}\frac{dt}{t} = \sum_{B \in \CB} g_{\alpha,\beta,\Phi}(f \ast \psi_B)(x)^2,
\end{align*}
and thus the decoupling estimate \eqref{decoupling} is proved.
\subsection{Re-coupling subdyadic frequency decompositions}
Here we prove the re-coupling estimate \eqref{recoupling} for the specific family of balls $\CB$ and partition $\{\widehat{\psi}_B\}$ described in Example \ref{Example}. While it may hold more generally, this lattice-based choice allows us to appeal to the following elementary lemma, which may be viewed as a certain local version of Bessel's inequality.
\begin{lemma}\label{EqSpaced}
Suppose that the functions $\nu_{k,\ell}\in\mathcal{S}(\mathbb{R}^d)$, with $k\in\mathbb{Z}$ and $\ell\in\mathbb{Z}^d$, are as in Example \ref{Example}. Then
\begin{equation}\label{bessel}
\sum_{\ell\in\mathbb{Z}^d}|f*\nu_{k,\ell}(x)|^2\lesssim|f|^2*|\nu_k|(x)
\end{equation}
uniformly in $k$.
\end{lemma}
\begin{proof}
By scaling it suffices to establish \eqref{bessel} with $k=0$. Noting that $\nu_0=\nu$,
observe that $f*\nu_{0,\ell}(x)=e^{2\pi i \ell\cdot x}\widehat{h}_x(\ell)$ where $h_x(y)=f(y)\nu(x-y)$. Hence by Parseval's identity, the Poisson summation formula and the Cauchy--Schwarz inequality,
\begin{eqnarray*}
\begin{aligned}
\sum_{\ell\in\mathbb{Z}^d}|f*\nu_{0,\ell}(x)|^2&=\int_{[0,1]^d}\Bigl|\sum_{\ell\in\mathbb{Z}^d}\widehat{h}_x(\ell)e^{2\pi i \ell\cdot y}\Bigr|^2dy=\int_{[0,1]^d}\Bigl|\sum_{m\in\mathbb{Z}^d}h_x(y+m)\Bigr|^2dy\\
&\leq\int_{[0,1]^d}\sum_{m\in\mathbb{Z}^d}|f(y+m)|^2|\nu(x-y-m)|\sum_{m'\in\mathbb{Z}^d}|\nu(x-y-m')|dy.
\end{aligned}
\end{eqnarray*}
Since $$\sum_{m'\in\mathbb{Z}^d}|\nu(x-m')|\lesssim 1$$ uniformly in $x$, we have
\begin{eqnarray*}
\begin{aligned}
\sum_{\ell\in\mathbb{Z}^d}|f*\nu_\ell(x)|^2&\lesssim \sum_{m\in\mathbb{Z}^d}\int_{[0,1]^d}|f(y+m)|^2|\nu(x-y-m)|dy
=|f|^2*|\nu|(x),
\end{aligned}
\end{eqnarray*}
as required.
\end{proof}
We may now establish the re-coupling estimate \eqref{recoupling} for the partition defined in Example \ref{Example}. For ease of notation we let $R_t^\lambda(x):=t^{(\alpha-1)d}(1+t^{\alpha-1}|x|)^{-d\lambda}$.
Observe first that since $\psi_B=\zeta_{k,\ell}=\eta_k*\nu_{k,\ell}$,
\begin{align*}
\sum_{B \in \CB} g_{\alpha,\beta,\lambda}^*(f \ast \psi_B)(x)^2 & =\sum_{k\in\mathbb{Z}}\sum_{\ell\in\mathbb{Z}^d}\int_{t^\alpha\leq 1}\int_{\mathbb{R}^d}
|f*\phi_t*\eta_k*\nu_{k,\ell}(y)|^2
R_t^\lambda(x-y) \frac{dy}{t^{2\beta}} \frac{dt}{t}\\
&=\int_{t^\alpha\leq 1}\int_{\mathbb{R}^d}
\sum_{2^k\sim t^{-1}}\sum_{\ell\in\mathbb{Z}^d}|f*\phi_t*\eta_k*\nu_{k,\ell}(y)|^2
R_t^\lambda(x-y) \frac{dy}{t^{2\beta}} \frac{dt}{t},
\end{align*}
where we have also used the Fourier support properties of $\widehat{\phi}_t$ to note that $\phi_t*\eta_k*\nu_{k,\ell}\neq 0$ only if $2^k\sim t^{-1}$. Applying Lemma \ref{EqSpaced}, followed by the Cauchy--Schwarz inequality, we have
$$
\sum_{\ell\in\mathbb{Z}^d}|f*\phi_t*\eta_k*\nu_{k,\ell}(y)|^2\lesssim |f*\phi_t*\eta_k|^2*|\nu_k|(y)\lesssim |f*\phi_t|^2*|\eta_k|*|\nu_k|(y)
$$
uniformly in $k$, $t$ and $y$, and hence by Fubini's theorem,
$$
\sum_{B \in \CB} g_{\alpha,\beta,\lambda}^*(f \ast \psi_B)(x)^2  \lesssim \int_{t^\alpha\leq 1}\int_{\mathbb{R}^d}
|f*\phi_t(y)|^2
\sum_{2^k\sim t^{-1}}|\eta_k|*|\nu_k|*R_t^\lambda(x-y) \frac{dy}{t^{2\beta}} \frac{dt}{t}.
$$
Observing the elementary inequality
$$
\sum_{2^k\sim t^{-1}}|\eta_k|*|\nu_k|*R_t^\lambda(x)\lesssim R_t^\lambda(x),
$$
which holds uniformly in $x$ and $t$ satisfying $t^\alpha \leq 1$, completes the proof of \eqref{recoupling}.

\subsection{The pointwise estimate at a subdyadic frequency scale}\label{subsec:localpointwise}
To conclude the proof of Theorem \ref{PointwiseThm} it is enough to show that
$$
g_{\alpha,\beta,2\sigma/d}^*(T_m(f \ast \psi_B))(x) \lesssim g_{\alpha,0,2\sigma/d}^*(f \ast \psi_B)(x)
$$
uniformly in $B\in\CB$. The argument we present is similar to that given in \cite{Stein70SI} in the classical setting. We begin by introducing an auxiliary function $\varphi_B$, chosen so that its Fourier transform is supported in $\widetilde{B}$ and is equal to $1$ on $\supp \widehat{\psi}_B$. For uniformity purposes we also assume, as we may, that
\begin{equation}\label{gencond2}|D^j \widehat{\varphi}_B(\xi)| \lesssim r(B)^{-|j|}
\end{equation}
for all multi-indices $j$, uniformly in $B\in\CB$. Observe that, up to a constant factor depending only on the uniform implicit constants in \eqref{gencond2}, $\widehat{\varphi}_B$ is a normalised bump function adapted to $B$. We begin by writing
\begin{align*}
g_{\alpha,\beta,2\sigma/d}^*(T_m(f \ast \psi_B))(x)^2 &=\int_{t^\alpha \leq 1} \int_{\R^d} |T_m(f \ast \varphi_B\ast\psi_B)\ast \phi_t(y)|^2 R_t^{2\sigma/d}(x-y) \frac{dy}{t^{2\beta}} \frac{dt}{t} \\
& = \int_{t^\alpha \leq 1} \int_{\R^d} |(T_m\varphi_B) \ast f \ast \psi_B \ast \phi_t(y)|^2 R_t^{2\sigma/d}(x-y) \frac{dy}{t^{2\beta}} \frac{dt}{t} \\
& \leq \int_{t^\alpha \leq 1} \int_{\R^d} \Big( \int_{\R^d} |T_m\varphi_B(z)| |f \ast \psi_B \ast \phi_t(y-z)|dz \Big)^2 R_t^{2\sigma/d}(x-y) \frac{dy}{t^{2\beta}} \frac{dt}{t}.
\end{align*}
For each $t$, we split the range of integration of the innermost integral in two parts, $|z|\leq t^{1-\alpha}$ and $|z| \geq t^{1-\alpha}$. For the term corresponding to $|z| \leq t^{1-\alpha}$, we use the Cauchy--Schwarz inequality, Plancherel's theorem, and the hypothesis \eqref{HormSobSD} with $\sigma=0$ to obtain
\begin{align*}
\Big(\int_{|z|\leq t^{1-\alpha}} |T_m\varphi_B(z)| |f \ast \psi_B \ast \phi_t(y-z)|dz\Big)^2 & \lesssim  t^{2\beta} t^{(\alpha-1)d} \int_{|z|\leq t^{1-\alpha}}|f \ast \psi_B \ast \phi_t(y-z)|^2dz \\
& \lesssim t^{2\beta} \int_{\R^d} R_t^{2\sigma/d}(z) |f \ast \psi_B \ast \phi_t(y-z)|^2dz;
\end{align*}
observe that the support hypotheses on $\widehat{\phi}$ and $\widehat{\psi}_B$ ensure $r(B)\sim t^{\alpha-1}$. Similarly, in $|z| \geq t^{1-\alpha}$,
\begin{align*}
\Big(\int_{|z|\geq t^{1-\alpha}} & |T_m\varphi_B(z)| |f \ast \psi_B \ast \phi_t(y-z)|dz\Big)^2 \\
&\leq \Big(\int_{\R^d} |T_m\varphi_B (z)|^2 |z|^{2\sigma} dz\Big) \Big(\int_{|z|\geq t^{1-\alpha}} \frac{1}{|z|^{2\sigma}}|f \ast \psi_B \ast \phi_t(y-z)|^2dz \Big) \\
&\lesssim t^{2\beta}t^{(\alpha-1)d}t^{2(1-\alpha)\sigma}\int_{|z|\geq t^{1-\alpha}} \frac{1}{(t^{1-\alpha}+|z|)^{2\sigma}}|f \ast \psi_B \ast \phi_t(y-z)|^2dz \\
& \lesssim t^{2\beta} t^{(\alpha-1)d} \int_{\R^d}(1+t^{\alpha-1}|z|)^{-2\sigma}|f \ast \psi_B \ast \phi_t(y-z)|^2dz.
\end{align*}
Putting together the above estimates we obtain
\begin{align*}
g_{\alpha,\beta,2\sigma/d}^*(T_m(f \ast \psi_B))(x)^2 & \lesssim \int_{t^\alpha \leq 1} \int_{\R^d} \Bigl(\int_{\R^d} R_t^{2\sigma/d}(z) |f \ast \psi_B \ast \phi_t(y-z)|^2 dz\Bigr) R_t^{2\sigma/d} (x-y) dy \frac{dt}{t} \\
& = \int_{t^\alpha \leq 1} \int_{\R^d} |f \ast \psi_B \ast \phi_t(y)|^2 R_t^{2\sigma/d} \ast R_t^{2\sigma/d} (x-y) dy \frac{dt}{t} \\
& \lesssim g_{\alpha,0,2\sigma/d}^*(f \ast \psi_B)(x)^2,
\end{align*}
where the last inequality follows since $\sigma>d/2$ and $R_t^\lambda \ast R_t^\lambda (x) \lesssim R_t^{\lambda}(x)$ for $\lambda>1$. This concludes the proof of Theorem \ref{PointwiseThm}.

\section{Proofs of Theorems \ref{ReverseThm} and \ref{ForwardThm}: the weighted inequalities}\label{sec:reverse}
Here we prove the reverse and forward weighted inequalities for the square functions $g_{\alpha,\beta}$ and $g_{\alpha,0,\lambda}^*$ respectively. In order to prove Theorem \ref{ReverseThm}, we make use of the following elementary lemma.
\begin{lemma}\label{L1Linf}
\begin{equation}\label{balls}
\int_{\R^d} f(x)h(x)dx \lesssim R^d \int_{\R^d} \int_{|y-x|\leq \frac{1}{R}}f(y)dy \sup_{z: |z-x|\leq \frac{1}{R}} h(z)\;dx
\end{equation}
uniformly in $R>0$ and nonnegative functions $f,h$ on $\R^d$.
\end{lemma}
\begin{proof}
To simplify notation, we prove the one-dimensional case of \eqref{balls}; the $d$-dimensional case follows by applying the one-dimensional in each variable and elementary geometric considerations. Observe that we may decompose the integral as
\begin{align*}
\int_\R f(x)h(x)dx&=\sum_{k \in \Z} \int_{-1/R}^{1/R}f\Big(x+u+\frac{2k}{R}\Big)h\Big(x+u+\frac{2k}{R}\Big)dx \\
& =\sum_{k \in \Z} \int_{|y-u-\frac{2k}{R}|\leq \frac{1}{R}}f(y)h(y)dy \\
& \leq \sum_{k \in \Z} \int_{|y-u-\frac{2k}{R}| \leq \frac{1}{R}} f(y)dy \sup_{z: |z-u-\frac{2k}{R}|\leq \frac{1}{R}}h(z)
\end{align*}
for any $|u|\leq \frac{1}{R}$. Averaging over $u$,
\begin{align*}
\int_\R f(x)h(x)dx & \leq \sum_{k \in \Z} 2R \int_{-1/R}^{1/R}\int_{|y-u-\frac{2k}{R}| \leq \frac{1}{R}} f(y)dy \sup_{z: |z-u-\frac{2k}{R}|\leq \frac{1}{R}}h(z)\; du \\
& =2R \sum_{k \in \Z} \int_{-1/R+2k/R}^{1/R+2k/R} \int_{|y-x| \leq \frac{1}{R}} f(y)dy \sup_{z: |z-x|\leq \frac{1}{R}}h(z)\; dx \\
& = 2R \int_\R \int_{|y-x| \leq \frac{1}{R}} f(y)dy \sup_{z: |z-x|\leq \frac{1}{R}}h(z)\; dx,
\end{align*}
as required.
\end{proof}

\begin{proof}[Proof of Theorem \ref{ReverseThm}]
We begin by using classical Littlewood--Paley theory in the  form of
\eqref{ReverseCont} to write
\begin{align*}
\int_{\R^d} |f(x)|^2 w(x)dx & \lesssim \int_0^\infty \int_{\R^d} |f \ast \phi_t(y)|^2 M^3w(y)dy \frac{dt}{t}.
\end{align*}
The support conditions on $\widehat{\phi}$ and $\widehat{f}$ (see the discussion in Section \ref{sec:Horm}) reduce the range for the $t$-integration to those $t$ such that $0<t^\alpha \leq 1$. Choosing $\varphi \in \CS$ such that $\widehat{\varphi}=1$ on the support of $\widehat{\phi}$ and $\supp(\widehat{\varphi})\subseteq \{\xi \in \R^d : \frac{1}{4} \leq |\xi| \leq 4\}$ allows us to write $f \ast \phi_t(y)=f \ast \phi_t \ast \varphi_t (y)$. Combining this with applications of the Cauchy--Schwarz inequality and Fubini's theorem gives
\begin{align*}
\int_{\R^d} |f(x)|^2 w(x)dx & \lesssim \int_{t^\alpha \leq 1} \int_{\R^d} |f \ast \phi_t(y)|^2 (|\varphi_t| \ast M^3w)(y) dy \frac{dt}{t} \\
& \lesssim \int_{t^\alpha \leq 1} \int_{\R^d} |f \ast \phi_t (y)|^2 A_t^*M^3w(y)dy \frac{dt}{t},
\end{align*}
where $A_t^* w(x):=\sup_{r \geq t}A_rw(x)$ and
$$
A_t w(x):=\frac{1}{|B(x,t)|}\int_{B(x,t)} w.
$$
Elementary considerations reveal that $A_t^*w \lesssim A_tA_t^*w \leq A_tMw$, so applying Lemma \ref{L1Linf} at scale
$R=t^{\alpha-1}$ yields
\begin{align*}
\int_{\R^d} |f(x)|^2 w(x)dx & \lesssim \int_{t^{\alpha}\leq 1} \int_{\R^d} \int_{|y-x|\leq t^{1-\alpha}} |f \ast \phi_t(y)|^2\frac{dy}{t^{(1-\alpha)d+2\beta}} \sup_{z: |z-x|\leq t^{1-\alpha}} t^{2\beta} A_t M^4 w(z)\; dx \frac{dt}{t} \\
& \leq \int_{\R^d} \int_{t^\alpha \leq 1} \int_{|y-x|\leq t^{1-\alpha}} |f \ast \phi_t(y)|^2 \frac{dy}{t^{(1-\alpha)d+2\beta}} \frac{dt}{t} \mathcal{M}_{\alpha,\beta}M^4w(x)dx,
\end{align*}
where the last inequality follows by taking the supremum in $t$, since
$$
\sup_{t^\alpha \leq 1} \sup_{z:|z-x|\leq t^{1-\alpha}} t^{2\beta} A_t M^4w(z) = \CM_{\alpha,\beta} M^4w(x),$$
by the definition of $\CM_{\alpha,\beta}$.
\end{proof}

Our proof of Theorem \ref{ForwardThm} is a simple consequence of the $L^2$-boundedness of more classical square functions of Littlewood--Paley type; see for example an analogous result for $g_\lambda^*$ given in \cite{Stein70SI}. We include the short argument here mainly for the sake of completeness.
\begin{proof}[Proof of Theorem \ref{ForwardThm}]
By Fubini's theorem,
\begin{align*}
\int_{\R^d} g_{\alpha,0,\lambda}^*(f)(x)^2w(x)dx
& = \int_{\R^d} \int_{t^\alpha \leq 1} |f \ast \phi_t(y)|^2 R_t^\lambda \ast w (y) dy \frac{dt}{t}.
\end{align*}
Since
$$
\sup_{t} R_t^\lambda \ast w \lesssim Mw
$$
for $\lambda>1$, we have
\begin{align*}
\int_{\R^d} g_{\alpha,0,\lambda}^*(f)(x)^2w(x)dx
\lesssim \int_{\R^d} \int_0^\infty |f \ast \phi_t(y)|^2 \frac{dt}{t} Mw(y) dy,
\end{align*}
which by an application of classical Littlewood--Paley theory in the form of \eqref{ForwardCont} results in
\begin{align*}
\int_{\R^d} g_{\alpha,0,\lambda}^*(f)(x)^2w(x)dx
\lesssim \int_{\R^d} |f(y)|^2 M^2w(y)dy,
\end{align*}
as required.
\end{proof}
As may be expected, it is possible to obtain similar weighted $L^2$ estimates for $g_{\alpha,\beta,\lambda}^*$ for other values of $\beta$ by minor modifications of the above argument.
\section{Proof of Corollary \ref{osccor}}\label{sec:Cor}
In \cite{Sjo81Lp} Sj\"olin establishes that the multiplier $\widehat{K}_{a,b}$ satisfies the Miyachi condition \eqref{MikhlinCondition}, leading to the conclusion
$$
g_{\alpha,\beta}(K_{a,b} \ast f)(x) \lesssim g_{\alpha,0,\lambda}^*(f)(x)
$$
for any $\lambda>0$, by a direct application of Theorem \ref{PointwiseThm}.

In order to prove \eqref{mainosc}, we must force the support condition on the multiplier $\widehat{K}_{a,b}$. We thus choose a function $\varphi \in C^\infty(\R^d)$ such that $\varphi(\xi)=0$ when $|\xi|^{\alpha} \leq 1$ and $\varphi(\xi)=1$ when $|\xi|^{\alpha} \geq 2$ and write $\widehat{K}_{a,b}=(1-\varphi)\widehat{K}_{a,b} + \varphi \widehat{K}_{a,b} = m_0+m_\infty$. The multiplier $m_\infty$ is supported in $\{\xi \in \R^d : |\xi|^\alpha \geq 1\}$ and satisfies the Miyachi condition \eqref{MikhlinCondition}, so Corollary \ref{MikhlinThm} immediately yields \eqref{mainosc} for $T_{m_\infty}$. The inequality for the portion $T_{m_0}$ follows from a straightforward adaptation of the techniques in Section \ref{sec:reverse}. Since $K_0=\widehat{m}_0$ is a rapidly decreasing function, the Cauchy--Schwarz inequality and Fubini's theorem allow us to write
$$
\int_{\R^d} |T_{m_0} f|^2w \lesssim \|K_0\|_1 \int_{\R^d} |f|^2 |K_{0}|\ast w \lesssim \int_{\R^d} |f|^2 A_1^*w \lesssim \int_{\R^d} |f|^2 M^2 \CM_{\alpha,\beta} M^4 w,
$$
where the last inequality follows from the pointwise bound
$$
A_1^*w \lesssim A_1 A_1^* w \lesssim \CM_{\alpha,\beta} A_1^*w \lesssim M^2 \CM_{\alpha,\beta} M^4 w.
$$
\section{Proof of Corollary \ref{PDEcorGlob}: applications to dispersive PDE}\label{sec:Dispersive} We shall prove Corollary \ref{PDEcorGlob} by establishing a more general statement for multipliers $m$ defined on the whole of $\R^d\backslash\{0\}$, satisfying the two-sided condition
\begin{equation}\label{twoside}
|D^\gamma m(\xi)|\lesssim
\begin{cases}
|\xi|^{-\beta+|\gamma|(\alpha-1)}, \:\:\:\: &\text{if} \:\:\:\: |\xi|^\alpha \geq 1
\\
|\xi|^{-\beta-|\gamma|}, \:\:\:\: &\text{if} \:\:\:\: |\xi|^\alpha \leq 1,
\end{cases}
\end{equation}
for all $\gamma \in \N^d$ such that $|\gamma| \leq \lfloor \frac{d}{2} \rfloor +1$; a natural example being the multiplier $m_{\alpha,\beta}(\xi):=|\xi|^{-\beta}e^{i|\xi|^\alpha}$. One might formulate a certain weaker two-sided H\"ormander-type condition similar to \eqref{HormSobSD} here, although we refrain from doing so for the sake of simplicity.
\begin{corollary}\label{mainGlobalcor}
If $m:\mathbb{R}^d\backslash\{0\}\rightarrow\mathbb{C}$ satisfies \eqref{twoside}
for all $\gamma \in \N^d$ such that $|\gamma| \leq \lfloor \frac{d}{2} \rfloor +1$, then
\begin{equation}\label{sleepy}\int_{\mathbb{R}^d}|T_mf|^2w\lesssim
\int_{\mathbb{R}^d}|f|^2M^2\mathfrak{M}_{\alpha,\beta}M^4w,
\end{equation}
where
$$
\mathfrak{M}_{\alpha,\beta}w(x)=\sup_{(y,r) \in \Lambda_\alpha(x)} \frac{1}{|B(y,r)|^{1-2\beta/d}} \int_{B(y,r)}w,
$$
and
$$
\Lambda_\alpha(x):=\{(y,r) \in \R^d \times \R_+ : |x-y|\leq r^{1-\alpha}\}.
$$
\end{corollary}
Corollary \ref{PDEcorGlob} follows from Corollary \ref{mainGlobalcor} by a straightforward scaling argument, the details of which we leave to the interested reader.
Before we turn to the proof of Corollary \ref{mainGlobalcor} we make some brief contextual remarks which illustrate the optimality of \eqref{sleepy} in the non-classical case $\alpha\neq 0$.

Corollary \ref{mainGlobalcor}, combined with the trivial $L^{d/(2\beta)}\rightarrow L^\infty$ bound on $\mathfrak{M}_{\alpha,\beta}$ and an elementary duality argument, quickly leads to the sharp $L^p\rightarrow L^q$ bounds for the class of multipliers satisfying \eqref{twoside}; see Miyachi \cite{Mi80}.
Hence for $\alpha\neq 0$, $\mathfrak{M}_{\alpha,\beta}$ necessarily fails to satisfy any other $L^p\rightarrow L^q$ inequalities. This is in contrast with the ``local" maximal functions $\mathcal{M}_{\alpha,\beta}$ associated with the regions $\Gamma_\alpha(x)$ studied in the previous sections, where $L^p\rightarrow L^q$ bounds exist with $q<\infty$; see the forthcoming Section \ref{sec:MaxOp}.

As with the classical fractional maximal operators, $\mathfrak{M}_{\alpha,\beta}$ behaves well on the power weights
$w_\gamma(x):=|x|^{-\gamma}$, with $0\leq \gamma<d$. Indeed it is straightforward to verify that
$\mathfrak{M}_{\alpha,\gamma/2} w_\gamma(x)\lesssim 1$ uniformly in $x\in\mathbb{R}^d$ and $\alpha\in\mathbb{R}$, and so Corollary \ref{mainGlobalcor} applied to the specific $m_{\alpha,\beta}$ reduces to
\begin{equation}\label{special}
\int_{\R^d} |e^{i(-\Delta)^{\alpha/2}}f(x)|^2|x|^{-\gamma}dx \lesssim \int_{\R^d} |(-\Delta)^{\gamma/4} f(x)|^2dx.
\end{equation}
This special case is somewhat degenerate as the presence of the parameter $\alpha$ is not detected in the estimate. It is instructive to prove \eqref{special} directly by an application of the classical Hardy inequality
$$
\int_{\R^d} |f(x)|^2|x|^{-\gamma}dx \lesssim \int_{\R^d} |(-\Delta)^{\gamma/4} f(x)|^2dx,
$$
followed by the energy conservation identity $\|e^{i(-\Delta)^{\alpha/2}}f\|_2=\|f\|_2$.
\begin{proof}[Proof of Corollary \ref{mainGlobalcor}]
Let $\eta\in C^\infty(\R^d)$ be such that $\eta(\xi)=0$ when $|\xi|^\alpha \leq 1$ and $\eta(\xi)=1$ when $|\xi|^\alpha \geq 2$, and write $m=m_1+m_2$, with $m_1=m\eta$ and $m_2=m(1-\eta)$. As $m_1$ is supported in $\{|\xi|^\alpha \geq 1\}$ and satisfies the Miyachi condition \eqref{MikhlinCondition}, Corollary \ref{MikhlinThm} gives
$$
\int_{\mathbb{R}^d}|T_{m_1}f|^2w\lesssim
\int_{\mathbb{R}^d}|f|^2M^2\mathcal{M}_{\alpha,\beta}M^4w.
$$
Similarly, the multiplier $m_2$ satisfies the condition \eqref{MikhlinCondition} for $\alpha=0$, so another application of Corollary \ref{MikhlinThm} gives
$$
\int_{\mathbb{R}^d}|T_{m_2}f|^2w\lesssim
\int_{\mathbb{R}^d}|f|^2M^2\mathcal{M}_{0,\beta}M^4w.
$$
As the maximal operator $\CM_{0,\beta}$ is pointwise comparable to the classical fractional Hardy--Littlewood maximal function of order $2\beta$,
$$
M_{2\beta}w(x)=\sup_{r>0} \frac{1}{r^{d-2\beta}}\int_{B(x,r)}w,
$$
one trivially has $\CM_{0,\beta} \lesssim \mathfrak{M}_{\alpha,\beta}$ for any $\alpha \in \R$. This, together with the obvious $\CM_{\alpha,\beta} \leq \mathfrak{M}_{\alpha,\beta}$, gives \eqref{sleepy} for $m_1$ and $m_2$, from which the result follows.
\end{proof}
We end this section by describing the analogue of Corollary \ref{PDEcorGlob} in the setting of inhomogeneous derivatives.
\begin{corollary}\label{inhomPDEthm}
\begin{equation}\label{weightedINH}
\int_{\R^d} |e^{is (-\Delta)^{\alpha/2}}f|^2w \lesssim \int_{\R^d} |(I-s^{2/\alpha}\Delta)^{\beta/2}f|^2 M^2 \mathcal{M}_{\alpha,\beta}^sM^4w,
\end{equation}
where
$$
\mathcal{M}_{\alpha,\beta}^sw(x):=\sup_{(y,r) \in \Gamma_\alpha^s(x)} \frac{r^{2\beta}}{|B(y,s^{1/\alpha}r)|}\int_{B(y,s^{1/\alpha}r)}w
$$
and
$$
\Gamma_\alpha^s(x)=\{(y,r) \in \R^{d}\times\R_+ : 0<r\leq 1, |x-y|\leq s^{1/\alpha}r^{1-\alpha}\}.
$$
\end{corollary}
Noting the elementary scaling identity
$$
e^{is(-\Delta)^{\alpha/2}}f(x)= 
T_{\widetilde{m}_{\alpha,\beta}} ((I-\Delta)^{\beta/2} f_s)(x/s^{1/\alpha}),
$$
where $\widetilde{m}_{\alpha,\beta}(\xi)=\frac{e^{i|\xi|^\alpha}}{(1+|\xi|^2)^{\beta/2}}$ and $f_s(x)=f(s^{1/\alpha}x)$, Corollary \ref{inhomPDEthm} follows from Corollary \ref{MikhlinThm} via an elementary rescaling argument; note that in order to apply Corollary \ref{MikhlinThm} to the multiplier $\widetilde{m}_{\alpha,\beta}$ it is necessary to consider its behaviour near the origin separately, as in the proof of Corollary \ref{osccor}.

Corollary \ref{inhomPDEthm} is optimal in the sense that the sharp bounds for $\mathcal{M}_{\alpha,\beta}^s$ given by the forthcoming Theorem \ref{MaxFunctionThm} allow one to recover the $W^{\beta,p}\rightarrow L^q$ bounds for $f\mapsto e^{is(-\Delta)^{\alpha/2}}f$ via \eqref{DualityArgument} in the sharp range of exponents; see Miyachi \cite{Mi81}. Here $W^{\beta,p}=(I-\Delta)^{-\beta/2}L^p$ denotes the classical inhomogeneous Sobolev space on $\mathbb{R}^d$.

\section{$L^p-L^q$ bounds} \label{sec:MaxOp}
In this section we establish the relevant $L^p-L^q$ bounds satisfied by the operators $g_{\alpha,\beta}$ and $\mathcal{M}_{\alpha,\beta}$. We begin by making the formal observation that if an inequality of the form
$$
\int_{\R^d} |f_1|^2 w \lesssim \int_{\R^d} |f_2|^2 \CM w
$$
holds for an $L^{(q/2)'}\rightarrow L^{(p/2)'}$ bounded operator $\mathcal{M}$, with $p,q\geq 2$, then by duality and H\"older's inequality it follows that
\begin{equation}\label{DualityArgument}
\|f_1\|_{q} \leq \|\CM\|^{1/2}_{L^{(q/2)'}-L^{(p/2)'}} \|f_2\|_{p},
\end{equation}
provided we known, a priori, that $\|f_1\|_q < \infty$. Thus, in particular, Lebesgue space bounds on the maximal function $\CM_{\alpha,\beta}$ allow one to deduce reverse bounds on the square function $g_{\alpha,\beta}$ via Theorem \ref{ReverseThm}. This of course leads to $L^p\rightarrow L^q$ bounds on multipliers satisfying \eqref{HormSobSD} via Theorem \ref{PointwiseThm} and \eqref{abstractbounds}, allowing one to reflect on the optimality of all such bounds.
\begin{theorem}\label{MaxFunctionThm}
Let $1<p\leq q\leq \infty$ and $\alpha,\beta \in \R$. If $\alpha>0$ and
$$
\beta \geq \frac{\alpha d}{2q} + \frac{d}{2}\Big(\frac{1}{p}-\frac{1}{q} \Big),
$$
or $\alpha=0$ and
$$
\beta=\frac{d}{2}\Big(\frac{1}{p}-\frac{1}{q} \Big),
$$
or $\alpha<0$ and
$$
\beta \leq \frac{\alpha d}{2q} + \frac{d}{2}\Big(\frac{1}{p}-\frac{1}{q} \Big),
$$
then $\CM_{\alpha,\beta}$ is bounded from $L^p(\R^d)$ to $L^q(\R^d)$.
\end{theorem}
This theorem is a straightforward adaptation of the one-dimensional case in \cite{Ben2014}; we give a sketch of the proof at the end of this section for the sake of completeness.

As discussed above, one may use Theorem \ref{MaxFunctionThm} to deduce reverse bounds for the square functions $g_{\alpha,\beta}$ via Theorem \ref{ReverseThm}.
\begin{corollary}\label{gbounds}
Let $\alpha,\beta \in \R$ and $2 \leq p \leq  q< \infty$. Let $f \in L^q(\R^d)$ be such that $\supp(\widehat{f}) \subseteq \{ \xi \in \R^d  : |\xi|^\alpha \geq 1\}$. If $\alpha>0$ and
$$\frac{\beta}{d} \geq \alpha\Big(\frac{1}{2}-\frac{1}{p} \Big) +\frac{1}{p}-\frac{1}{q},$$
or $\alpha=0$ and
$$\frac{\beta}{d} = \frac{1}{p}-\frac{1}{q},$$
or $\alpha<0$ and
$$\frac{\beta}{d} \leq \alpha\Big(\frac{1}{2}-\frac{1}{p} \Big) +\frac{1}{p}-\frac{1}{q},$$
then
$$
\|f\|_q \lesssim \|g_{\alpha,\beta}(f)\|_p.
$$
\end{corollary}
Corollary \ref{gbounds} provides reverse Lebesgue space estimates for square functions that decouple frequency decompositions that are typically much finer than dyadic. Forward estimates for much more singular subdyadic square functions were obtained by Rubio de Francia \cite{RdF}, who established the uniform $L^p(\R)$ boundedness of the square function
$$
Sf(x)=\Big(\sum_{I \in \CI} |\Delta_I f(x)|^2 \Big)^{1/2}
$$
for $2\leq p < \infty$. Here $\CI$ denotes an arbitrary collection of disjoint intervals and $\widehat{\Delta_I f}(\xi)=\chi_I(\xi)\widehat{f}(\xi)$; see Journ\'e \cite{Journe} for a similar $d$-dimensional result. As remarked in the introduction, one may view Corollary \ref{gbounds} as a substitute for the known failure of the reverse $L^p$ bounds for $S$ when $p\geq 2$.

Of course the multipliers \eqref{HormSobSD} are bounded from $L^p(\R^d)$ to $L^q(\R^d)$ for the same exponents $(p,q)$ as in Corollary \ref{gbounds}, as the square function $g_{\alpha,0,\lambda}^*$ trivially satisfies $\|g_{\alpha,0,\lambda}^*(f)\|_p \lesssim \|f\|_p$ for $2 \leq p < \infty$. This recovers a number of well-known multiplier theorems since our class \eqref{HormSobSD} naturally contains those considered by Miyachi \cite{Mi81} -- in addition to the classical H\"ormander--Mikhlin multipliers and fractional integrals. In particular, as the model multipliers $|\xi|^{-\beta}e^{i|\xi|^\alpha}\chi_{\{|\xi|^\alpha\geq 1\}}$ are bounded on $L^p(\R^d)$ if and only if $|1/2-1/p| \leq \beta/(\alpha d)$, the maximal operators $\CM_{\alpha,\beta}$ in \eqref{mainweight} are optimal in the sense that they cannot be replaced by variants satisfying additional Lebesgue space bounds.

\begin{proof}[Proof of Theorem \ref{MaxFunctionThm}]
We only concern ourselves with the case $\alpha \neq 0$; the case $\alpha=0$ corresponds to the classical fractional Hardy--Littlewood maximal function. Observe that for $\alpha>0$, the possible radii $r$ in the approach region $\Gamma_\alpha(x)$ satisfy $0<r\leq 1$ and therefore
$$\CM_{\alpha,\beta'}w \leq \CM_{\alpha,\beta}w$$
for $0<\beta < \beta'$. A similar analysis for the case $\alpha<0$ reveals that it is enough to show that $\CM_{\alpha,\beta}$ is bounded from $L^p(\R^d)$ to $L^q(\R^d)$, where $1<p \leq q \leq \infty$, on the line
\begin{equation}\label{sharpline}
\beta = \frac{\alpha d}{2q} + \frac{d}{2}\Big(\frac{1}{p}-\frac{1}{q} \Big).
\end{equation}
We regularise the average in the definition of $\CM_{\alpha,\beta}$ and we prove the estimates for the pointwise larger maximal operator
$$\widetilde{\CM}_{\alpha,\beta}w(x)=\sup_{(y,r)\in \Gamma_\alpha(x)}r^{2\beta}|P_r \ast w(y)|,$$
where $P$ is a nonnegative compactly supported bump function which is positive on $B(0,1)$ and $P_r(x):=r^{-d}P(x/r)$. Trivially,
$$
|\widetilde{\CM}_{\alpha,0}w(x)|=\sup_{(y,r)\in \Gamma_\alpha(x)}|P_r \ast w(y)| \leq \|P\|_1 \|w\|_\infty
$$
and
$$
|\widetilde{\CM}_{\alpha, \frac{d}{2}}w(x)|=\sup_{(y,r)\in \Gamma_\alpha(x)}r^d|P_r \ast w(y)| \leq \|P\|_\infty \|w\|_1,
$$
for every $x \in \R^d$. Analytic interpolation between these two estimates gives $\widetilde{\CM}_{\alpha,\frac{d}{2p}}:L^p(\R^d) \to L^\infty(\R^d)$ for $1\leq p \leq \infty$. A further application of analytic interpolation shows that boundedness of $\widetilde{\CM}_{\alpha,\beta}$ from $L^p(\R^d)$ to $L^q(\R^d)$ holds for $\alpha,\beta$ as in \eqref{sharpline} provided $\widetilde{\CM}_{\alpha,\frac{\alpha d}{2}}: H^1(\R^d) \to L^1(\R^d)$. To this end, it suffices to see
$$
\|\widetilde{\CM}_{\alpha,\frac{\alpha d}{2}}a\|_1 \lesssim 1
$$
uniformly in $H^1(\R^d)$-atoms $a$; recall that an atom $a$ is a function defined on $\R^d$ supported in a cube $Q$ such that $\int_Q a=0$ and $\|a\|_\infty \leq \frac{1}{|Q|}$. By translation-invariance, we may assume that the cube $Q$ is centered at the origin, and thus we may appeal to the standard bounds
\begin{equation*}
|P_r \ast a(y)| \lesssim \begin{cases}
1/|Q| & \text{ if }  r<|Q|^{1/d}, |y|\lesssim 2|Q|^{1/d} \\
|Q|^{1/d}/r^{d+1} & \text{ if } r>|Q|^{1/d}, |y|\lesssim 2r \\
0 & \text{ otherwise.}
\end{cases}
\end{equation*}
As the nature of the region $\Gamma_\alpha$ changes dramatically for $\alpha<0$, $0<\alpha \leq 1$ and $\alpha>1$, an accurate case analysis for the different values of $\alpha$ completes the proof; see \cite{Ben2014} for further details when $d=1$, which extend routinely to higher dimensions.
\end{proof}

\section{Appendix: Weighted $L^2$ estimates for classical square functions}
Let $\phi$ be a smooth function as discussed in Section \ref{sec:Horm}. Consider the square function
$$
s_\phi(f)(x)=\Big(\int_0^\infty |f \ast \phi_t(x)|^2\frac{dt}{t} \Big)^{1/2}.
$$
\begin{proposition}
\begin{equation}\label{ForwardCont}
\int_{\R^d} s_\phi(f)(x)^2 w(x)dx \lesssim \int_{\R^d} |f(x)|^2 Mw(x)dx
\end{equation}
and
\begin{equation}\label{ReverseCont}
\int_{\R^d} |f(x)|^2 w(x)dx \lesssim \int_{\R^d} s_\phi(f)(x)^2 M^3w(x)dx.
\end{equation}
\end{proposition}
\begin{proof}
The forward inequality \eqref{ForwardCont} is a straightforward consequence of a more general result of Wilson \cite{Wi07}. To prove \eqref{ReverseCont}, we adapt a simple argument from \cite{BH}. Condition \eqref{ContinuousPartition} allows one to write $f=f_0+\cdots+f_5$, where
$$
f_j(x)= \sum_{k \in 6\mathbb{Z}+\{j\}} \int_{2^k}^{2^{k+1}} \phi_t \ast f(x) \frac{dt}{t}.
$$
As $\supp(\widehat{\phi}) \subseteq \{\xi \in \R^d: 1/2 \leq |\xi| \leq 2\}$, the function $\xi \to \int_1^2 \widehat{\phi}(t\xi)\frac{dt}{t}$ is supported in $\{\frac{1}{4} \leq |\xi| \leq 2\}$. Let $\epsilon=(\epsilon_k)$ be a random sequence with $\epsilon_k \in \{-1,1\}$. For each $j=0,\dots,5$, we define $T_j^\epsilon$ by
$$
\widehat{T_j^\epsilon f}(\xi)=\sum_{k \in 6\mathbb{Z}+\{j\}} \epsilon_k \chi(2^{-k} \xi) \widehat{f}(\xi),
$$
where $\chi$ is a smooth function that equals 1 in $\{\frac{1}{4} \leq |\xi| \leq 2\}$ and vanishes outside $\{\frac{1}{8} \leq |\xi| \leq 4\}$. Using elementary support-disjointness considerations, it is clear that $T_j^\epsilon T_j^\epsilon f_j=f_j$, and so
$$
\int_{\R^d} |f(x)|^2 w(x)dx \lesssim \sum_{j=0}^5 \int_{\R^d} |T_j^\epsilon T_j^\epsilon f_j(x)|^2 w(x)dx. 
$$
Next we appeal to the elementary fact that $T_j^\epsilon$ is a standard Calderón--Zygmund operator with kernel bounds that may be taken to be uniform in the sequence $\epsilon$, along with a classical weighted estimate for Calderón--Zygmund operators (see Theorem 3.4 in \cite{Wi89} or Theorem 1.1. in \cite{Pe94}), to deduce that
$$
\int_{\R^d} |f(x)|^2 w(x)dx \lesssim \sum_{j=0}^5 \int_{\R^d} |T_j^\epsilon f_j(x)|^2 M^3w(x)dx,
$$
uniformly in $\epsilon$. 
Taking expectations, and using Khintchine's inequality results in
$$
\int_{\R^d} |f(x)|^2 w(x)dx \lesssim \sum_{j=0}^5 \int_{\R^d} \sum_{k \in 6\Z +\{j\}} \Big| \int_{2^k}^{2^{k+1}} \phi_t \ast f(x) \frac{dt}{t} \Big|^2 M^3w(x)dx,
$$
which together with the Cauchy--Schwarz inequality allows us to conclude that
$$
\int_{\R^d} |f(x)|^2 w(x)dx \lesssim \sum_{j=0}^5 \int_{\R^d} \sum_{k \in 6\Z + \{j\}} \int_{2^k}^{2^{k+1}} |\phi_t \ast f(x)|^2 \frac{dt}{t} M^3(x)dx = \int_{\R^d} s_\phi(f)(x)^2 M^3w(x)dx.
$$
\end{proof}

\bibliographystyle{abbrv} 

\bibliography{BeltranBennettarXivAccepted} 
%
%
%

\end{document}